\documentclass[leqno,11pt]{article}
\usepackage{microtype}

\usepackage{amsmath,amsthm,amssymb,mathrsfs} 
\usepackage[T1]{fontenc}
\usepackage{ae,aecompl}
\usepackage{times}
\usepackage[english]{babel}
\usepackage[colorlinks,pagebackref,hypertexnames=false]{hyperref}
\usepackage[alphabetic,backrefs]{amsrefs}
\usepackage{bookmark}

\numberwithin{equation}{section}

\newcommand{\rG}{{\rm G}}

\renewcommand{\epsilon}{\varepsilon}

\newcommand{\tr}{\mathop{\mathrm{tr}}\nolimits}

\def\<{\mathopen{}\left<}
\def\>{\right>\mathclose{}}
\def\({\mathopen{}\left(}
\def\){\right)\mathclose{}}

\newtheorem{theorem}{Theorem}

\newtheorem{corollary}{Corollary}

\newtheorem{definition}{Definition}

\newtheorem{lemma}{Lemma}

\newtheorem{problem}{Problem}
\newtheorem{proposition}{Proposition}
\newtheorem{remark}{Remark}

\numberwithin{equation}{section}

\author{Goncalo Oliveira \\ Duke University}

\title{Gerbes on $G_2$ Manifolds}

\date{March 2015}

\begin{document}
\maketitle

\begin{abstract}
On a projective complex manifold, the Abelian group of Divisors maps surjectively onto that of holomorphic line bundles (the Picard group). On a $G_2$-manifold we use coassociative submanifolds to define an analogue of the first, and a gauge theoretical equation for a connection on a gerbe to define an analogue of the last. Finally, we construct a map from the former to the later. Finally, we construct some coassociative submanifolds in twisted connected sum $G_2$-manifolds.
\end{abstract}


\section{Introduction}

In a complex manifold $X$ one defines the Abelian group of divisors, $Div(X)$, as formal sums of complex codimension $1$ submanifolds. On the other hand the holomorphic line bundles also form an Abelian grop, known as the the Picard group $Pic(X)$. It is a classical fact that $Pic(X) \cong Div(X)/ \sim$, where $\sim$ denotes linear equivalence of divisors. If $X$ is supposed to have a K\"ahler form, it follows from Hodge theory that each holomorphic line bundle has a unique Hermitian-Yang-Mills (HYM) connection, equivalently a unique connection with harmonic curvature. In this short note, inspired by Hitchin's work on the moduli of special Lagrangian submanifolds, we imitate part of these ideas from complex geometry to the case of $G_2$ manifolds.\\

Let $(X^7 , g)$ be a compact $G_2$-manifold, i.e. a compact, Riemannian $7$-manifold with holonomy contained in $G_2$. Equivalently, one can think of the metric $g$ has being induced from a stable $3$-form $\varphi$, which is harmonic. By this reason we shall also refer to a $G_2$-manifold as the pair $(X^7, \varphi)$. A $G_2$-manifold is said to be irreducible if in addition its holonomy representation is irreducible, i.e. $Hol(g)=G_2$. For future reference we shall use the notation $\psi = \ast \varphi \in \Omega^4(X, \mathbb{R})$, where $\ast$ is the Hodge-$*$ operator of the Riemannian metric $g$.\\
On a $G_2$-manifold there are very interesting calibrated submanifolds called associatives and coassociates. The latest are the main subject of this paper and so we now focus on them. These are $4$-dimensional submanifolds $N^4$ calibrated by $\psi$, i.e. $\psi \vert_N = dvol_{g\vert_N}$, where $g\vert_N$ denotes the restriction of the metric $g$ to $N$. We use them to define the following $G_2$ analogue of the group of divisors

\begin{definition}
Let $CDiv(X,\varphi)$ denote the Abelian group of finite formal sums 
$$\sum_{i=1}^k q_i N_i,$$
where the $q_i \in \mathbb{Z}$ and the $N_i$ are coassociative submanifolds.
\end{definition}


In \cite{Joyce2012} Joyce have conjectured that it may be possible to define an invariant of $G_2$-manifolds "counting" coassociative submanifolds. In \cite{Donaldson2009} Donaldson and Segal suggested it may be easier, to define an invariant from solutions of a gauge theoretical equation, the $G_2$-monopole equation. The authors have further suggested that such monopoles may be somehow related to coassociative submanifolds, and this have been further investigated in \cite{Oliveira2014}. This short note contains a detour, where we consider gauge theoretical objects of another nature, which show to also be related to coassociative submanifolds. Before we dive into their definition let us give some more motivation.\\
Given a coassociative submanifold $N^4$, its Poincar\'e dual gives a cohomology class $H^3(X, \mathbb{Z})$. Conversely, one can ask the following questions: Given $\alpha \in H^3(X, \mathbb{Z})$, is there a coassociative representative of $\alpha$? What are the possible topological obstructions for a class $\alpha$ to be represented by a coassociative submanifold? (for an integrable $G_2$-structure as $H^4_7=0$ we do not know any such). Given a class $\alpha$, are there any toplogical/geometric restrictions of possible coassociative representatives of $\alpha$?
Some of these problems are currently out of reach and is certainly not the author's goal to attempt it. Nevertheless, we believe our results may be of interest and we hope may motivate further research along similar lines. We give some possibilities for these at the end of this note.\\

We turn now to the definitions of the gauge theoretical objects, which are our $G_2$ analogues of the Picard group. Recall that we want these to be related to coassociative submanifolds, which are of codimension $3$, this is where gerbes with connection come into play. For the sake of simplicity, in this introduction it is enough to think of a gerbe as a \v Cech cocycle $\mathcal{G} \in \check{H}^2(X, C^{\infty}(S^1))$, where $C^{\infty}(S^1)$ is the sheaf of smooth functions with values in $S^1$. This should be compared with the way a cocycle in $\check{H}^1(X, C^{\infty}(S^1))$ gives a circle bundle. Gerbes, as bundles can be equipped with connections. There is the notion of a trivialization of a gerbe, and if $\lbrace U_i \rbrace_{i \in I}$ is an open cover of $X$, such that on each $U_i$ we have gerbe trivializations, then a connection is determined by connection $2$-forms $F_i$ satisfying some extra conditions, see section \ref{sec:Gerbes} for more details. These conditions ensure that the exterior derivatives $dF_i$ agree on double intersections and so define a global, closed $3$-form $H$ called the curvature of the connection.\\
The long exact sequence induced by the exponential gives an isomorphism 
$$c_1: \check{H}^2(X, C^{\infty}_X(S^1)) \xrightarrow{\sim} H^3(X, \mathbb{Z}).$$
The image $c_1(\mathcal{G})$ of a gerbe $\mathcal{G}$ under such a map is its first Chern class, and the curvature $H$ of any connection on $\mathcal{G}$ lies in the cohomology class $c_1(\mathcal{G})$. We are now ready to define our $G_2$-analogue of the Picard group

\begin{definition}\label{def:MonopoleGerbe}
A \textbf{monopole gerbe} is a pair $(\mathcal{G}, F)$, where $\mathcal{G}$ is a gerbe and $F$ a connection on $\mathcal{G}$, such that
\begin{itemize}
\item There is an open cover $\lbrace U_{i} \rbrace_{i \in I}$ of $X$ equipped with local trivializations of $\mathcal{G}$, and connection $2$-forms $F_i$ satisfying
$$\ast (F_i \wedge \psi) = d \phi_i,$$
for $\phi_i$ some real valued functions.
\item The curvature of $F$ is the harmonic representative of $c_1(\mathcal{G})$.
\end{itemize}
The monopole gerbes form an Abelian group, which we denote $MPic(X, \varphi)$.
\end{definition}

The operation under which the monopole gerbes form an Abelian group is the tensor product of the underlying gerbes and the canonically induced connections, see section \ref{sec:MonopoleGerbes} below.

\begin{remark}
We compare each of the conditions above with the definition of an holomorphic line bundle, or rather a complex line bundle with a connection inducing an holomorphic structure
\begin{itemize}
\item The first condition is the analogue of the existence of holomorphic trivializations where the connection $1$-forms are of type $(1,0)$.
\item Requiring the curvature of a gerbe connection to be the harmonic representative of $c_1(\mathcal{G})$ is the analogue of in the case of an holomorphic bundle changing hermitian structure so that the connection has harmonic curvature. In the K\"ahler case there is a unique such connection (up to gauge) and this is the so called Hermitian-Yang-Mills one (HYM).
\end{itemize}
\end{remark}

In the general context of a Riemannian manifold, Hitchin defined in \cite{Hitchin99} a way to canonically associate a gerbe with connection to a codimension $3$ submanifold. We review this construction in proposition \ref{Connection}. We are now ready to state our

\begin{theorem}\label{thm:1}
Let $(X^7, \varphi)$ be a compact, irreducible $G_2$-manifold, then Hitchin's construction yields a group homomorphism
\begin{equation}\label{eq:Map}
m: CDiv(X, \varphi) \rightarrow MPic(X, \varphi).
\end{equation}
\end{theorem}

There is a forgetful map $MPic(X, \varphi) \rightarrow H^3 (X, \mathbb{Z})$, which associates to $(\mathcal{G},F)$ the topological type of the gerbe $c_1(\mathcal{G})$. We prove the following analogue of the hard-Lefschetz theorem for $(1,1)$ classes

\begin{theorem}\label{thm:2}
Let $(X, \varphi)$ be a compact, irreducible $G_2$-manifold, then the map $MPic(X, \varphi) \rightarrow H^3 (X, \mathbb{Z})$ is surjective.
\end{theorem}
 
Putting theorem \ref{thm:2} together with \ref{thm:1} we see that, if the map \ref{eq:Map} is surjective, then the $G_2$-analogue of the Hodge conjecture holds true. However, as we said before it is not our goal to attempt an answer to that question.\\

On the other hand, a more practical goal with the tools at hand, is to obtain restrictions on the geometry/topology of possible coassociative submanifolds of certain classes. This is of interest has giving topological restrictions for the convergence of the mean curvature flow for a $4$ dimensional submanifolds of $G_2$-manifolds. One other possible place of interest is for example in understanding the global geometry of coassociative fibrations on a $G_2$ manifold. Our final result (which is not an application of the previous ones) is aligned with that direction, giving restrictions on the topological type of coassociatives depending on the cohomology class they represent. This follows from a coassociative analogue of the adjunction formula and the fact that on a compact, non-flat $G_2$-manifold $p_1(X) \cup [\varphi] >0$. Namely we prove that

\begin{theorem}\label{prop:3}
Let $(X, \varphi)$ be a compact $G_2$-manifold, $\alpha \in H^{3}(X, \mathbb{Z})$, and $N$ a coassociative representative of $\alpha$. If $\tau$ and $\chi$ respectively denote the signature and Euler characteristic of $N$, then $p_1(X) \cup \alpha = 6 \tau -2 \chi$. In particular,
\begin{itemize}
\item $\tau > \frac{1}{3} \chi$ if and only if $p_1(X) \cup \alpha >0$;
\item $\tau = \frac{1}{3} \chi$ if and only if $p_1(X) \cup \alpha =0$;
\item $\tau < \frac{1}{3} \chi$ if and only if $p_1(X) \cup \alpha <0$.
\end{itemize}
\end{theorem}

\begin{remark}\label{rem:Signature0}
In this result I have used the convention that my $3$-forms are modeled on
$$\varphi_0= e^{123} + e^1 \wedge (e^{45} -e^{67})+ e^2 \wedge (e^{46} -e^{75})+ e^3 \wedge (e^{47} -e^{56}),$$
on $\mathbb{R}^7$. In this way, $N=0 \times \mathbb{R}^4$ is coassociative and the map $e_i \mapsto \iota_{e_i} \varphi$ yields an isomorphism of $TN^{\perp}$ with $\Lambda^2_-N$ the anti-self-dual $2$-forms on $N$. For such irreducible $G_2$-manifolds locally modeled on these we have $p_1(X) \cup [\varphi]>0$ , see remark \ref{rem:Signature}. However, many authors use $G_2$-structures locally modeled on $\varphi_0= e^{123} + e^1 \wedge (e^{45} +e^{67})+ e^2 \wedge (e^{46} +e^{75})+ e^3 \wedge (e^{47} +e^{56})$, in which case the normal bundle of a coassociative submanifold $N$ is isomorphic to $\Lambda^2_+N$. For such structures, we have $p_1(X) \cup [\varphi] < 0$ on any irreducible $G_2$-manifold. Moreover, the result of theorem \ref{prop:3} should then be stated as $p_1(X) \cup \alpha = 6 \tau +2 \chi$ and so
\begin{itemize}
\item $\tau > -\frac{1}{3} \chi$ if and only if $p_1(X) \cup \alpha >0$;
\item $\tau = -\frac{1}{3} \chi$ if and only if $p_1(X) \cup \alpha =0$;
\item $\tau < -\frac{1}{3} \chi$ if and only if $p_1(X) \cup \alpha <0$.
\end{itemize}
In fact, these two statements are related by changing the orientation of the coassociative, in which case is calibrated by $- \psi$.
\end{remark}

An immediate consequence of this result is that given any irreducible $G_2$-manifold $(X,\varphi)$ and a compact $4$-manifold $N^4$, there are classes in $X$, such that $N$ cannot be embedded in $X$ as a coassociative representative of such classes. For example, given that a $4$-torus has $\tau=\frac{\chi}{3}=0$ we have

\begin{corollary}
Any coassociative $T^4$ of a compact $G_2$-manifold must represent a class $\alpha$ such that $p_1(X) \cup \alpha =0$. In particular, if $b^3(X) = 1$, then $(X, \varphi)$ has no coassociative tori.
\end{corollary}

We must however remark that up to the author's knowledge there are no known examples of $G_2$-manifolds with $b^3(X)=1$.\\

In the rest of this note we recall some basic notions of gerbes in section \ref{sec:Gerbes}. Namely we recall Hitchins working definition for gerbes and connections on them. Then, we recall Hitchin's construiction of a gerbe with connection associated with a codimension-$3$ submanifold. In section \ref{sec:MonopoleGerbes} we give the proofs of the results mentioned in this introduction. In section $4$ we give an illustrative example in the reducible case. Also, in that section we give some constructions of coassociative submanifolds of twisted connected sums.

\subsection*{Acknowledgments}

I would like to thank Robert Bryant, Simon Donaldson, Hans Joachim Hein, Spiro Karigiannis, Johannes Nordstr\"om and Mark Stern for discussions regarding the result in this short note. I also want to thank an anonymous referee for comments and the AMS and Simons foundation for travel support.

\section{Gerbes}\label{sec:Gerbes}

In the introduction we defined a \textbf{gerbe} as a C\^ech cocycle $\mathcal{G} \in \check{H}^2(X, C^{\infty}(S^1))$. Alternatively, \cite{Hitchin99} one can take an open cover $\lbrace U_{\alpha} \rbrace_{\alpha \in I}$ and a gerbe $\mathcal{G}$ can be defined by the following data
\begin{itemize}
\item A line bundle $L_{\alpha \beta}$ over each $U_{\alpha \beta}=U_{\alpha} \cap U_{\beta}$, and isomorphisms $L_{\beta \alpha} \cong L_{\alpha \beta}^{-1}$.

\item Trivializations $\theta_{\alpha \beta \gamma} : L_{\alpha \beta}L_{\beta \gamma} L_{\gamma \alpha} \cong \underline{\mathbb{C}}$, such that $\delta \theta =1$ on $U_{\alpha \beta \gamma \delta}$.
\end{itemize}
Using this point of view, one a \textbf{connection} $F$ on $\mathcal{G}$ is determined by the following data
\begin{itemize}
\item Connections $\nabla_{\alpha \beta}$ on the $L_{\alpha \beta}$, such that $\nabla_{\alpha \beta \gamma} \theta_{\alpha \beta \gamma}=0$. 

\item $2$-forms $F_{\alpha}$ on the $U_{\alpha}$, such that
$$F_{\alpha \beta}=F_{\alpha}-F_{\beta},$$
is the curvature of $\nabla_{\alpha \beta}$ on $L_{\alpha \beta}$.
\end{itemize}
Then, in the double intersections $U_{\alpha \beta}=U_{\alpha} \cap U_{\beta}$
$$dF_{\beta} = dF_{\alpha} + d F_{\alpha \beta}= dF_{\alpha},$$
by the Bianchi identity. Therefore, there is a well defined $3$-form $H$ such that
$$H \vert_{U_{\alpha}} = dF_{\alpha},$$
for all $\alpha \in I$. This $H$ is called the \textbf{curvature} of the gerbe connection.

\begin{remark}
If $(\mathcal{G},F)$ is a monopole gerbe, then it is easy to see that the connections $\nabla_{\alpha \beta}$ on the line bundles satisfy
$$F_{\alpha \beta} \wedge \psi = \ast d \phi_{\alpha \beta},$$
where $F_{\alpha \beta}=F_{\alpha}-F_{\beta}$ and $\phi_{\alpha \beta}= \phi_{\alpha}-\phi_{\beta}$. In other words $(\nabla_{\alpha \beta}, \phi_{\alpha \beta})$ form an Abelian monopole on $L_{\alpha \beta}$. This justifies our nomenclature.
\end{remark}

\subsection*{Codimension-$3$ submanifolds and gerbes}

In this section we shall consider a more general setup where $(X^n,g)$ is a real $n$-dimensional Riemannian manifold and $N$ a codimension $3$ (embedded) submanifold. The next proposition is an analogue of the construction of the map $Div \rightarrow Pic$ in complex geometry and I learned it for gerbes from Hitchin's paper \cite{Hitchin99}. For completeness, we shall include the construction.

\begin{proposition}\label{Connection}
Let $N$ be a codimension $3$, connected and embedded submanifold of $X$ and $H \in PD[N] \in H^3(X, \mathbb{Z})$. Then, there is a Gerbe with connection $(\mathcal{G}_{H}, F)$ whose curvature is $H$. In particular, $c_1(\mathcal{G}_{H}) = PD \left[ N \right]$.
\end{proposition}
\begin{proof}
To construct the gerbe take a finite open cover given by $U_0 = X \backslash N$ and $\lbrace U_{\alpha} \rbrace_{\alpha \in I}$, such that $N \subset \cup_{\alpha \in I} U_{\alpha}$ and each $U_{\alpha} \cap N$, $U_{\alpha} \cap U_{\beta}$ is contractible. We shall use the indices $\lbrace i, j, k \rbrace$ to refer to either $0$ or $\alpha \in I$. To define the Gerbe $\mathcal{G}_{H}$, one must give line bundles $L_{ij}$ on the double intersections $U_{ij} = U_i \cap U_j$ satisfying a cocycle condition on the triple intersections. Using the notation $L_{ij}=L_{ji}^{-1}$, this is given by fixing a trivialization $L_{ij} \otimes L_{jk} \otimes L_{ki} \cong \underline{\mathbb{C}}$ on the triple intersections $U_{ijk}= U_i \cap U_j \cap U_k$.\\
Notice that up to homotopy $U_{0 \alpha} \cong ( U_{\alpha}\cap N) \times \mathbb{S}^2$, then we let $L_{\alpha 0}$ be the pullback of Hopf bundle on the $\mathbb{S}^2$ factor. On $U_{\alpha \beta}$ we let $L_{\alpha \beta}$ be the trivial bundle. Then the cocycle condition is trivially satisfied on the the triple intersection $U_{0 \alpha \beta}$ by fixing a trivialization $\underline{\mathbb{C}} \cong L_{0\alpha} \otimes L_{\beta 0} \vert_{U_{0 \alpha \beta}}$.\\
We turn now to the definition of the connection. This requires giving the connection $2$-forms $F_i$ on each $U_i$, such that
$$F_{ij}= F_i - F_j,$$
is the curvature of a connection on $L_{ij}$. We weakly solve the PDE for currents
$$\Delta H_0 = H- \delta_{N}$$
which is possible, since $\left[ H- \delta_N \right]=0$ in de Rham cohomology for currents. Then $dH_0=0$, as $\Delta dH_0=d \Delta H_0=0$ being both exact and harmonic and so vanishes. Also notice that $H_0$ is only unique up to an harmonic $3$-form. It is also possible to solve
$$\Delta H_{\alpha} = H \ \ , \ \ dH_{\alpha}=0,$$
on each open set $U_{\alpha}$. Using the solutions $H_0, H_{\alpha}$ to these equations, one can define the connection $2$-forms by
\begin{eqnarray}\nonumber
F_0 & = & d^* H_0 \ , \ \text{on $U_0$} \\ \nonumber
F_{\alpha} & = & d^* H_{\alpha} \ , \ \text{on $U_{\alpha}$}.
\end{eqnarray}
Notice that $F_0$ is indeed uniquely defined as any other $H_0$ will differ by a global harmonic three form, which is then coclosed and give rise to the same $F_0$. One still needs to check that the $2$-forms $F_{ij}$ are the curvature of a connection on $L_{ij}$. For $F_{\alpha \beta}$ this is obvious as $dF_{\alpha \beta} = d(  F_{\alpha}-F_{\beta} )=0$ on $U_{\alpha \beta}$ and the Poincar\'e lemma gives a primitive to $F_{\alpha \beta}$ which we take to be our connection on $L_{\alpha \beta}$.\\
Over $U_{0\alpha}$ there is a unique nontrivial $2$ cycle, namely the one generated by the $2$-spheres $S^2$ in the normal bundle. These do bound a $3$-dimensional disk $D^3$ in $U_{\alpha}$ but not in $U_{0 \alpha}$, as any such $D^3$ does need to intersect $N$. Since $dF_{\alpha 0}= d(F_{\alpha}- F_{0} )= \delta_{N}$, Stokes' theorem gives
$$\int_{S^2}F_{0\alpha} = \int_{D^3} dF_{0 \alpha} = \int_{D^3} \delta_N =1.$$
This shows that the $2$-forms $F_i$ do define a connection on the gerbe $\mathcal{G}_{N}$. To check that its curvature is $H$, just compute $dF_0 = dd^* H_0 = H$ in $U_0$ and $dF_{\alpha} =d d^* H_{\alpha} = H$ in $U_{\alpha}$.
\end{proof}

Let $Y$ denote the disjoint union of a collection of open sets covering $X$. Then, the connections on the gerbes $\mathcal{G}_N$ above where defined using $2$-forms $F \in \Omega^2(Y)$ satisfying some compatibility conditions, also known as curvings \cite{Murray2010}. Given two open coverings $Y_1 = \lbrace U^1_{\alpha} \rbrace_{\alpha \in I}$ and $Y_2 = \lbrace U^2_{\beta} \rbrace_{\beta \in J}$, we can define a refined open cover
$$Y_{12}= \lbrace U^1_{\alpha} \cap U^2_{\beta} \rbrace_{(\alpha, \beta) \in I \times J}.$$
Then, we can add two curvings $F^1 \in \Omega^2(Y_1)$ and $F^2 \in \Omega^2(Y_2)$ to define a curving $F^{12} \in \Omega^2(Y_{12})$, such that
$$F^{12} \vert_{U^1_{\alpha} \cap U^2_{\beta}} = F^1 \vert_{U^1_{\alpha}} + F^2 \vert_{ U^2_{\beta}}.$$ 
Moreover, this can be taken inductively to define $Y_{1...k}$, for any $k$-tuple of open covers and curvings $F^{1...k} \in \Omega^2(Y_{1...k})$

\begin{definition}\label{Product}
Let $N= N_1 \cup ...  \cup N_k$ with each $N_i$ an embedded, connected submanifold, and for $i=1,...,k$ let $H_i \in PD[N_i]$ be $3$-forms representing the respective cohomology class. The previous proposition \ref{Connection} gives gerbes $(\mathcal{G}_{N_i}, F_i)$ with connections defined by curvings $F_i \in \Omega^2(Y_i)$ whose curvature is $H_i$. We define the gerbe with connection $(\mathcal{G}_N, F)$ to be given by $\mathcal{G}_N= \mathcal{G}_{N_1}...\mathcal{G}_{N_k}$ and the connection by the curving $F^{1...k} \in \Omega^2(Y_{1...k})$.
\end{definition}

In the definition above $\mathcal{G}_N= \mathcal{G}_{N_1}...\mathcal{G}_{N_k}$ denotes the tensor product of the gerbes. The resulting gerbe has first Chern class $c_1(\mathcal{G}_N)= c_1(\mathcal{G}_{N_1}) + ... + c_1(\mathcal{G}_{N_k}) = PD[N]$ and indeed $(\mathcal{G}_N, F)$ does have curvature $H=H_1 + ... + H_k$.

\begin{lemma}\label{ProductLemma}
In the setup of definition \ref{Product}, let all $N_i$'s be embedded and such that $N= N_1 \cup ... \cup N_k$ remains embedded. Then, given $H \in PD[N]$ the gerbe with connection $(\mathcal{G}_{N}, F)$ from proposition \ref{Connection} coincides with the one constructed via definition \ref{Product}, for any choice of $H_i \in PD[N_i]$, such that $H= \sum_{i=1}^k H_i$.
\end{lemma}
\begin{proof}
For simplicity we shall only do the case $k=2$, the gerbes in question are defined via line bundles on the double intersections of the open cover given by the sets $U_{0}^{1} = X\backslash N_1$, $U_{0}^{2}=X\backslash N_2$ and $\lbrace U_{\alpha}^{1} \rbrace_{\alpha \in I_1}$, $ \lbrace U_{\alpha}^{2} \rbrace_{\alpha \in I_2}$ such that $N_i \subset \cup_{\alpha \in I_i} U^i_{ \alpha}$, for $i=1,2$ and we suppose the $U^1_{\alpha}$'s are disjoint from the $U^2_{\alpha}$'s. As before, let $U_0 = X\backslash N$ and weakly solve the PDE's for the following $3$-forms (currents)
$$\Delta H_{\alpha}^i = H_i \ \ , \ \ dH_{\alpha}^{i}=0,$$
on each $U^i_{ \alpha}$ and
$$\Delta H_{0}^{i} = H_{i}- \delta_{P_i}, $$
on $U_{0}^{i}$ for $i=1,2$. Then, the $2$-forms defining the connection are given by
\begin{eqnarray}\nonumber
F_0 & = &  d^* (H_{0}^{1}+H_{0}^{2}) \ , \ \text{on $U_0$} \\ \nonumber
F_{\alpha}^{1} & = & d^* (H_{\alpha}^{1} + H_{0}^{2}) \ , \ \text{on $U_{\alpha}^{1}$} \\ \nonumber
F_{\alpha}^{2} & = & d^*( H_{\alpha}^{2} + H_{0}^{1}) \ , \ \text{on $U_{\alpha}^{2}$}.
\end{eqnarray}
It follows that $dF_0= H_1 + H_2$ on $U_0$ and for $i=1,2$ one has $dF_{\alpha}^{i}= H_1 + H_2 $ on each $U_{\alpha}^{i}$, which shows that the curvature of the connection on $\mathcal{G}_{N_1} \otimes \mathcal{G}_{N_2}$ is $H=H_1 +H_2$. To check that the connection does not depend on the splitting $H=H_1 + H_2$ take instead the splitting given by the $3$-forms $H_1 + d \omega$ and $H_2-d \omega$ for some $\omega \in \Omega^2(X)$. These are obviously cohomologous to the initial ones and do add to $H$. The forms $H^i_{0}, H^i_{\alpha}$ change by $+ \alpha, - \alpha$, for $i=1,2$ respectively, where $\alpha$ satisfies
$$dd^* \alpha = d \omega \ \ , \ \ d\alpha =0.$$
So the forms $F^i_{0}= d^* H^i_{0}$ and $F^i_{\alpha} = d^*H^i_{\alpha}$ change by $\pm d^* \alpha $ and so $F_0 = F_{0}^1 + F_0^2$ do not change. Exactly the same argument works for the $F_{\alpha}$'s and we conclude that the connection remains unchanged.
\end{proof}

\section{Proof of the main results}\label{sec:MonopoleGerbes}

We now go back to the case when $(X, \varphi)$ is a $G_2$-manifold and denote by $\psi= \ast \varphi$ the calibrating $4$-form. The $3$-forms in a $G_2$ manifold, pointwise split into $G_2$-irreducible representations as
$$\Lambda^3 = \Lambda^3_1 \oplus \Lambda^3_7 \oplus \Lambda^3_{27},$$
where the subscripts in the right hand side denote the respective dimension. We shall use $\pi_1,\pi_7,\pi_{27}$ to denote the respective projections.\\
 Recall from definition \ref{def:MonopoleGerbe} that a gerbe with connection $(\mathcal{G}, F)$ is said to be a monopole gerbe if its curvature $H \in \Omega^3(X, \mathbb{R})$ is the harmonic representative of $c_1(\mathcal{G}) \in H^3(X, \mathbb{Z})$ and there is a trivialization $\lbrace U_{\alpha} \rbrace_{\alpha \in I}$ and $2$-forms $F_{\alpha}$ satisfying
$$\ast (F_{\alpha} \wedge \psi ) = d \phi_{\alpha},$$
where the $\phi_{\alpha}$'s are real valued functions.\\
It is now easy to see that under the tensor product of the gerbes and addition of the curvings defining the connections, the monopole gerbes define an Abelian group which we have denoted by $MPic(X, \varphi)$.
We shall now see that associated to a coassociative submanifold there is a canonical monopole gerbe.

\subsection{Proof of theorem \ref{thm:1}}

Before we dive into the proof let us prove an easy but key lemmata.

\begin{lemma}\label{lem:Coassociative}
Let $N$ be a coassociative submanifold and $\delta_N$ the current it generates, then $\delta_N \wedge \varphi=0$.
\end{lemma}
\begin{proof}
An equivalent way to define a coassociative submanifold is to say that $\varphi \vert_N =0$. Then for all $\eta \in \Omega^1(X)$,
$$\delta_N \wedge \varphi (\eta) = \int_{N} \eta \wedge \varphi =0.$$
\end{proof}

\begin{lemma}\label{lem:2}
Let $(F, \phi)$ be a $2$-form and a function on a contractibe open set $U$ of a $G_2$-manifold. Then, if $\pi_7 dF=0$ and $F \wedge \psi = \ast d\phi$, we have $F=d^*G$ where $G$ is a closed $3$-form.
\end{lemma}
\begin{proof}
We write $F= \ast (f \wedge \psi) + g$ for some $1$-form $f$ and $g \in \Lambda^2_{14}$. Then, as $F\wedge \psi=\ast d\phi$, we conclude that $f= \frac{d \phi}{3}$ and using table $3$ in \cite{Bryant2006} we compute $dF = -\frac{1}{7} \Delta (\frac{\phi}{3}) \varphi + \pi_{7} dg + \pi_{27} dg$. Moreover, as $\pi_7(dF)=0$ by assumption, we conclude that $\pi_7 dg=0$.\\
Now, we compute $d\ast F =\ast (\pi_7 dg \wedge \varphi)=0$ and so $F$ is coclosed. As $U$ is contractible we can write $F=d^* G''$ on $U$ and extend $G''$ to a $3$-form $G'$ on $X$, for example by multiplying $G''$ by a bump function supported on a slightly larger open set $U'\supset U$ and letting it vanish on its complement. Then, by de Rham's theorem $G'=da_2 + d^*a_4 + a_3$ for some $a_i \in \Omega^i$ with $a_3$ harmonic. We now redefine $G=da_2 \vert_U$, which is therefore closed (in fact exact) and on $U$ 
$$d^*G=d^* d a_2 = d^*G'=d^* G''=F,$$
as we wanted to show.
\end{proof}

\begin{theorem}\label{th:MonopoleGerbe}
Let $(X, \varphi)$ have holonomy strictly equal to $G_2$, $N$ be a connected, embedded, coassociative submanifold. Then, the gerbe with connection associated with $N$ is a monopole gerbe. 
\end{theorem}
\begin{proof}
Let $H$ be the harmonic representative of $PD[N]$ and recall the construction of the connection $F$ on $\mathcal{G}_N$. On the open cover $\lbrace U_{i} \rbrace_{i \in \lbrace 0 \rbrace \cup I}$ it is given by a collection of $2$-forms $F_i$. These are defined such that each $F_i = d^* H_i$, with $H_i$ a collection of $3$-forms such that $\Delta H_0=H- \delta_N$ and $\Delta H_{\alpha} = H$, $dH_{\alpha}=0$ for $\alpha \in I$.\\
We shall first work on the open set $U_0$. It follows from the fact that $H$ is the harmonic representative of $PD[N]$ that it has no component in $\Lambda^3_7$. Moreover, lemma \ref{lem:Coassociative} guarantees that $\delta_N$ also has no component in $\Lambda^3_7$. Hence, as on a $G_2$-manifold the Laplacian preserves the type decomposition and $\Delta H_0 = H- \delta_{N}$, it follows that $\pi_7(H_0)$ is harmonic. As $X$ is supposed to have full $G_2$ holonomy, there can be no parallel $1$-forms, \cite{Bryant1989}, and so the B\"ochner-formula implies that $\pi_7(H_0)=0$. Using this the equation $dH_0=0$ turns into $d \pi_1(H_0) = - d \pi_{27}(H_0)$ and if one writes $\pi_1 (H_0) = -a \varphi$, for some function $a$, this is
\begin{equation}\label{dH27}
d\pi_{27}(H_0)=  da \wedge \varphi.
\end{equation}
Then we compute $d^* \pi_{27}H_0= \pi_7 d^* \pi_{27} H_0 + \pi_{14}d^* \pi_{27} H_0$, using that $\pi_7 d^* \pi_{27} H_0 = - \frac{1}{3} \ast (\ast ( \ast d \pi_{27} H_0 \wedge \varphi ) \wedge \psi )$ together with equation \ref{dH27} gives
\begin{eqnarray}\nonumber
d^* H_0 & = & \pi_{14}d^* \pi_{27} H_0 -  \frac{1}{3} \ast (\ast ( \ast d \pi_{27}H_0 \wedge \varphi ) \wedge \psi ) - d^* (a \varphi) \\ \nonumber
& = &  \pi_{14}d^* \pi_{27} H_0 +\frac{4}{3} \ast \left( d a \wedge \psi \right)  + \ast ( d a \wedge\psi ) \\ \nonumber
& = & \pi_{14}d^* \pi_{27} H_0 + \frac{7}{3}  \ast \left( d a \wedge \psi \right) ,
\end{eqnarray}
where we used that $\ast (\ast (d a \wedge \varphi) \wedge \varphi) = -4 d a$. Now we put $F_0=d^* H_0$, $\phi = 7a$ and compute $\ast (F_0 \wedge \psi )$. Since $\Omega^2_{14}$ is the kernel of wedging with $\psi$ and $\ast ( \ast (d \phi \wedge \psi ) \wedge \psi) = 3 d \phi$, we obtain
\begin{equation}
\ast ( F_0 \wedge \psi ) = d \phi.
\end{equation}


We now need to define the connection $2$-forms for our monopole gerbe on the remaining $U_{\alpha}$'s. Since $H$ is closed we can locally find on the $U_{\alpha}$'s $2$-forms $F_{\alpha}'$ such that $dF_{\alpha}'=H$. Restricting to each $U_{\alpha}$ we shall seek a connection $2$-form $F_{\alpha}=F_{\alpha}' + da_{\alpha}$ such that the monopole equation $\ast (F_{\alpha} \wedge \psi) = d \phi_{\alpha}$ holds on $U_{\alpha}$ for some $\phi_{\alpha}$.\\
To do this we set $(a_{\alpha}, \phi_{\alpha})=(\ast(db_{\alpha} \wedge \psi), -d^{\ast} b_{\alpha})$ and solve for $b_{\alpha}$ instead. Using $3d^7 \cdot = \ast(\ast(d \cdot \wedge \psi) \wedge \psi)$ and that $3d^*d^7=d^* d$, the monopole equation turns into
$$-\ast(F_{\alpha}' \wedge \psi) = 3d^*d^7 b_{\alpha} + dd^* b_{\alpha} = \Delta b_{\alpha}.$$
Moreover, as $g$ is Ricci flat, on $1$-forms $\Delta = \nabla^* \nabla$ and so we need to solve $\nabla^* \nabla b_{\alpha}=- \ast (F_{\alpha} \wedge \psi)$. This can be done by solving the Dirichelet problem
\begin{eqnarray}\nonumber
\nabla^* \nabla b_{\alpha} & = & - \ast (F_{\alpha}' \wedge \psi), \text{ on $U_{\alpha}$} \\ \nonumber
b_{\alpha} \vert_{\partial U} & = & 0 , \text{ on $\partial U_{\alpha}$}.
\end{eqnarray}
This follows from minimizing the functional $J(u)= \int_U \vert \nabla u \vert^2+\langle u , f \rangle$, where $f=\ast (F_{\alpha}' \wedge \psi)$. To prove it is coercive we proceed as follows
\begin{eqnarray}\nonumber
J(u) & = & \int_U \vert \nabla u \vert^2 + \langle u , f \rangle \\ \nonumber
& \geq & \int_U \vert \nabla \vert u \vert \vert^2 - \frac{\epsilon}{2} \int_U \vert u \vert^2 - \frac{1}{2 \epsilon} \int_U \vert f \vert^2 \\ \nonumber
& \geq & \left( c_U - \frac{\epsilon}{2}  \right) \int_U \vert u \vert^2 -\frac{1}{2 \epsilon} \int_U \vert f \vert^2.
\end{eqnarray}
where $c_U>0$ is some constant and $\epsilon>0$ is to be chosen small enough to ensure the first term is positive. We also remark that in the computation above, the first inequality follows from Kato's and Young's inequalities, while the second one makes use of Poincar\'e's inequality, as $u$ has vanishing boundary values. This shows that the functional $J$ is convex and so there is a unique solution to the Dirichlet problem above. From this procedure we back to set $F_{\alpha}=F_{\alpha}'+d a_{\alpha}$ and $(a_{\alpha},\phi_{\alpha})=(\ast(db_{\alpha} \wedge \psi), -d^{\ast} b_{\alpha})$, then by construction $(F_{\alpha}, \phi_{\alpha})$ satisfy the monopole equation, and it is immediate that $dF_{\alpha}=H$. Moreover, it follows from lemma \ref{lem:2} that the connection $2$-forms constructed in this way agree with the one from the construction in proposition \ref{Connection}.

To check that this collection of connection $2$-forms give a trivialization of the claimed monopole gerbe connection one needs to check that each $F_{0\alpha}=F-F_{\alpha}$ is indeed the curvature of a connection on the bundles $L_{0\alpha}$. This is exactly the same as in standard construction in proposition \ref{Connection} and so we omit it. 
\end{proof}

In fact, as remarked to the author by an anonymous referee, the second part of the previous proof also proves that given a harmonic $3$-form $H$ on a compact, irreducible $G_2$-manifold, then there is a monopole gerbe whose curvature is $H$. The author, does not know whether a gerbe with harmonic curvature, on a compact, irreducible $G_2$-manifold can be written as a monopole gerbe without changing the connection (up to gauge). If true, this is analogous to the fact that in an irreducible Calabi-Yau manifold of complex dimension greater than $2$, any line bundle with harmonic curvature is holomorphic. A slightly more general version of this "integrability" theorem is stated as problem \ref{conj:integrability}.\\
The second part of the previous proof is, in fact also the key for proving theorem \ref{thm:2}

\begin{proof}
We prove that the map $MPic(X, \varphi) \rightarrow H^3(X, \mathbb{Z})$, $\mathcal{G} \mapsto c_1(\mathcal{G})$ is surjective. To do this we start with a gerbe $\mathcal{G}$ and construct a connection $F$ satisfying the required conditions. Let $H \in c_1(\mathcal{G})$ be the harmonic representative and $\lbrace U_{\alpha} \rbrace_{\alpha \in I}$ be a good open cover of $X$, i.e. each $U_{\alpha}$ and $U_{\alpha \beta}$ is contractible.\\
Then, following the last step in the proof of theorem \ref{th:MonopoleGerbe} we let $F_{\alpha}'$ be such that $dF_{\alpha}'=H$. Then, correct each of these to $F_{\alpha}=F_{\alpha}' + d a_{\alpha}$, with $(a_{\alpha}, \phi_{\alpha})=(\ast(db_{\alpha} \wedge \psi), -d^{\ast} b_{\alpha})$, such that $F_{\alpha} \wedge \psi = \ast d\phi_{\alpha}$. As before, this holds if and only if $\nabla^* \nabla b_{\alpha}=- \ast (F_{\alpha} \wedge \psi)$ which can be solved by standard minimization techniques as in the previous proof.
\end{proof}

\begin{remark}
We point out that theorem \ref{thm:2} and its proof show that any gerbe with harmonic curvature can be tensored with a flat gerbe so that the resulting gerbe with connection is a monopole gerbe.\\
We also point out that through Hitchin's construction the gerbe associated with any codimension-$3$ manifold has harmonic curvature. However, only in the case when this submanifold is coassociative does Hitchin's construction yield a monopole gerbe.
\end{remark}

\subsection*{Proof of theorem \ref{prop:3}}

Let $p_1(X)$ denote the first Pontryagin class on the tangent bundle of $X$. If $N \subset X$ is a coassociative submanifold, then $TX \vert_N \cong TN \oplus \Lambda^2_-(N)$ and so $p_1(X)\vert_N = p_1(N) + p_1(\Lambda^2_-(N))$. Moreover, if $N$ represents a class $\alpha \in H^3(X, \mathbb{Z})$, we have
\begin{equation}\label{eq:Inequality}
\langle p_1(X) \cup \alpha , [X] \rangle = \int_N p_1(N) + \int_N p_1(\Lambda^2_-(N)).
\end{equation}
The first term is $3 \tau$ and we shall now compute the second. Fix a local orthonormal framing $\lbrace e_0, e_1,e_2,e_3  \rbrace$. Then in this framing, the curvature of the Levi-Civita connection on $N$ acts via the matrix
$$ R=\begin{pmatrix}
0 & \Omega^0_1 & \Omega^0_2 & \Omega^0_3 \\
- \Omega^0_1 & 0 & \Omega^1_2 & \Omega^1_3 \\
-\Omega^0_2 & - \Omega^1_2 & 0 & \Omega^2_3  \\
-\Omega^0_3 & - \Omega^1_3 & - \Omega^2_3 & 0 
\end{pmatrix} \in \Omega^2(N, \mathfrak{so}(TN)) ,$$
where $\Omega^i_j(X,Y)=e^i(R(X,Y) e_j)$. Using this, the Gauss-Bonnet formula and Chern-Weil theory give
\begin{eqnarray}\nonumber
\chi & = & \frac{1}{2^4 \pi^2 2!} \sum_{ijkl} \epsilon_{ijkl} \Omega^i_j \wedge \Omega^k_l = \frac{1}{2^3 \pi^2 } \sum_{ijk} \epsilon_{ijk} \Omega^0_i \wedge \Omega^j_k \\ \nonumber
& = & \frac{1}{4 \pi^2} \left( \Omega^0_1 \wedge \Omega^2_3 + \Omega^0_2 \wedge \Omega^3_1 + \Omega^0_3 \wedge \Omega^1_2 \right). \\ \nonumber
p_1(N) & = & - \frac{1}{8 \pi^2} \tr (R \wedge R) \\ \nonumber
& = & \frac{1}{4 \pi^2} \left( \Omega^0_1 \wedge \Omega^0_1 + \Omega^0_2 \wedge \Omega^0_2 + \Omega^0_3 \wedge \Omega^0_3 +  \Omega^1_2 \wedge \Omega^1_2 + \Omega^1_3 \wedge \Omega^1_3 + \Omega^2_3 \wedge \Omega^2_3 \right).
\end{eqnarray}
Now let $\omega_1=e_{01}-e_{23}$, $\omega_2=e_{02}-e_{31}$, $\omega_3=e_{03}-e_{12}$ be a local basis of $\Lambda^2_-$. The curvature of the induced connection on $\Lambda^2_-$ acts on this basis by
$$ R_{\Lambda^2_-}=\begin{pmatrix}
0 & -(\Omega^0_3 - \Omega^1_2) & \Omega^0_2 + \Omega^1_3  \\
\Omega^0_3-\Omega^1_2 & 0 & -(\Omega^0_1-\Omega^2_3)  \\
-(\Omega^0_2 + \Omega^1_3) & \Omega^0_1-\Omega^2_3 & 0 
\end{pmatrix} \in \Omega^2(N, \mathfrak{so}(\Lambda^2_-)) ,$$
Then, we compute
\begin{eqnarray}\nonumber
p_1(\Lambda^2_-N) & = & - \frac{1}{8 \pi^2} \tr (R_{\Lambda^2_-} \wedge R_{\Lambda^2_-}) \\ \nonumber
& = & \frac{1}{4 \pi^2} \left( \Omega^0_1 \wedge \Omega^0_1 + \Omega^0_2 \wedge \Omega^0_2 + \Omega^0_3 \wedge \Omega^0_3 +  \Omega^1_2 \wedge \Omega^1_2 + \Omega^1_3 \wedge \Omega^1_3 + \Omega^2_3 \wedge \Omega^2_3 \right) \\ \nonumber
& & - \frac{1}{2 \pi^2} \left( \Omega^0_1 \wedge \Omega^2_3 + \Omega^0_2 \wedge \Omega^3_1 + \Omega^0_3 \wedge \Omega^1_2 \right)\\ \nonumber
& = & p_1(N)-2 \chi.
\end{eqnarray}
Hence, inserting this into equality \ref{eq:Inequality} and using the signature theorem we have
\begin{equation}\label{eq:Inequality2}
\int_X p_1(X) \cup \alpha = \langle (2p_1(N)-2 \chi ) , [N] \rangle = 6 \tau - 2 \chi
\end{equation}
as we wanted to prove.

\begin{remark}\label{rem:Signature}
\begin{enumerate}
\item Using our conventions for the $3$-form $\varphi$, see remark \ref{rem:Signature0}, we have $p_1(X) \cup [\varphi] >0$ for any compact nonflat $G_2$-manifold. To see this we notice that $p_1(X) \cup \varphi = - \frac{1}{8 \pi^2} \tr (F_R \wedge F_R) \wedge \varphi$. Moreover, as $X$ has holonomy $G_2$, $F_R$ takes values in $\Lambda^2_{14} \cong \mathfrak{g}_2$ by the Ambrose-Singer theorem. Then, using our conventions we have $F_R \wedge \varphi = \ast F_R$, so that
$$\langle p_1(X) \cup \varphi , [X] \rangle = - \int_X \frac{1}{8 \pi^2} \tr (F_R \wedge \ast F_R) = \Vert F_R \Vert^2_{L^2} >0,$$
as $(X, \varphi)$ is not flat.\\
Using the other convention for $\varphi$ we have $F_R \wedge \varphi = -\ast F_R$ and so this sign gets reversed. In any case, there do exist classes $\alpha^+, \alpha^- \in H^3(X, \mathbb{R})$ such that $p_1(X) \cup \alpha^+ >0$ and $p_1(X) \cup \alpha^- <0$.

\item In \cite{McLean1998}, corollary $4-3$ McLean observed that as the torus $\mathbb{T}^7$ has trivial tangent bundle, $p_1(\mathbb{T}^7)=0$ and any coassociative submanifold of $\mathbb{T}^7$ must have $\tau= \chi/3$ (there is a typo in the statement). For instance $N=\mathbb{T}^4 \times 0 \subset \mathbb{T}^7$ has $\tau=0=\chi$ and satisfies such equality.

\item On a $G_2$-manifold there is one other class of very interesting submanifolds known as associatives. These are defined by requiring that the restriction of $\varphi$ to them agrees with the volume form of the induced metric. Then, the following even easier argument shows that there are no associative submanifolds $M$ of a compact, irreducible $G_2$-manifold $(X, \varphi)$ representing the class $p_1(X)$. Otherwise we would have $vol(M)=\int_M \varphi = \int_X p_1(X) \cup [\varphi] < 0$, which is clearly impossible.
\end{enumerate}
\end{remark}

\section{Toy examples}

In this section we explore the construction above in a case when the holonomy representation is actually reducible and in order to adapt the definition of monopole gerbe to this case we require that $\pi_7(H)=0$, where $H$ is the harmonic curvature. Let $X= S^1 \times M^6$ with $(M, \omega , \Omega)$ being an irreducible Calabi-Yau with K\"ahler form $\omega$ and holomorphic volume form $\Omega=\Omega_1 + i \Omega_2$. In this case the $G_2$-structure is
$$\varphi = d \theta \wedge \omega - \Omega_1, \ \psi = -d \theta \wedge \Omega_2 + \frac{\omega^2}{2},$$
where $\theta$ is a periodic coordinate on $\mathbb{S}^1$. From, the last of these formulas it is easy to identify two types of coassociative submanifolds of $S^1 \times M$. Namely, those calibrated by either $-d\theta \wedge \Omega_2$, or $\frac{\omega^2}{2}$. These are of the form $S^1 \times SL^3$ and ${pt.} \times \mathcal{D}^4$, where $SL$, $\mathcal{D}$ and $pt.$ are respectively a special Lagrangian submanifold of $M$, a divisor in $M$ and a point in $S^1$.\\
In order to analyze some features of our construction in this case we need to explain how to get a complex line bundle $\pi_* \mathcal{G}$ over $M$ from a gerbe $\mathcal{G}$ over $S^1 \times M$. This is most naturally seen by regarding $\mathcal{G}$ as a line bundle over the loop space $L(M \times S^1)$ as in \cite{Hitchin99} and chapter 6 in \cite{Brylinski2009}. Each element of $L(M \times S^1)$ is a map $\gamma: S^1 \rightarrow M \times S^1$. As $S^1$ is $1$-dimensional, the pulled back gerbe $\gamma^* \mathcal{G}$ has a flat trivialization given by a flat line bundle $L_{\gamma} \rightarrow S^1$, and trivializations whose difference is a flat bundle with trivial holonomy are regarded as equivalent. Then, the moduli space of flat connections on $S^1$, i.e. $H^1(S^1,\mathbb{Z}) = S^1$ acts on these trivializations by tensoring with a flat connection. This shows that the space
$$P= \lbrace (\gamma, L_{\gamma}) \ \vert \ \text{$\gamma \in L(M\times S^1)$ and $L_{\gamma}$ is a trivialization of $\gamma^* \mathcal{G}$} \rbrace,$$
together with the $S^1$ action described above is a circle bundle over $L(M\times S^1)$. Now we consider the map
\begin{eqnarray}
M \xrightarrow{\Gamma} L(S^1 \times M), \ p \mapsto \gamma_p,
\end{eqnarray}
where $\gamma_p : S^1 \rightarrow S^1 \times M$ is the loop $\gamma_p(\theta)=(\theta, p)$.

\begin{definition}
We define $\pi_*\mathcal{G}$ to be the complex line bundle associated with $\Gamma^* P$ over $M$.
\end{definition}

\begin{proposition}
Let $D^4 \subset M$ be a $4$-dimensional submanifold and $\mathcal{G}$ be the gerbe associated with $pt. \times D^4 \subset S^1 \times M$. If $\mathcal{G}$ is a monopole gerbe, then $D$ is a divisor. In particular $\pi_*\mathcal{G}$ comes equipped with an HYM connection.
\end{proposition}
\begin{proof}
Let $H$ be the harmonic representative of $N=pt. \times D$. As the Poincar\'e dual of $N$ is $[d\theta] \cup PD[D]$, this is of the form $H = d \theta \wedge h$ where $h \in \Omega^2(M, \mathbb{Z})$ is the harmonic representative of $PD[D]$. Similarly $\delta_M = \delta(\theta-pt) d \theta \wedge \delta_D$ and we solve the equation
\begin{equation}\label{eq:ProductEq}
\Delta H_0= (h- \delta_D \delta(\theta-pt)) \wedge d \theta,
\end{equation}
which implies $dH_0=0$. Using the splitting $\Lambda^3 X = \Lambda^3 M \oplus \Lambda^1 S^1 \otimes \Lambda^2 M$ we write $H_0=g(\theta) + d\theta \wedge f(\theta)$ and compute
\begin{eqnarray}\label{eq:fIsClosed}
d_M f & = & \frac{\partial g}{\partial \theta} \\  \nonumber
d_M g & = & 0 \\ \label{eq:ProductEq2}
\Delta H_0 & = & d\theta \wedge \left( \Delta_M f - \frac{\partial^2 f}{\partial \theta^2} \right) +  \Delta_M g - \frac{\partial^2 g}{\partial \theta^2} 
\end{eqnarray}
So the equation $\Lambda^3 M$ component of equation \ref{eq:ProductEq} turns into $ \Delta_M g - \frac{\partial^2 g}{\partial \theta^2} =0$. Separation of variables plus periodicity in $\theta$ then imply that $g$ must be constant and $g_M$-harmonic, and equation \ref{eq:fIsClosed} turns into $d_M f=0$.\\
Now define the map $\pi_* : \Omega^{*+1}(M \times S^1) \rightarrow \Omega^*(M)$, given by $\pi_* (\omega)=0$ and $\pi_*(d\theta \wedge \omega)= \int_{S^1} d\theta \wedge \omega$, for $\omega \in \Lambda^{k} M$. Let $h_0=\pi_* H_0$, then applying $\pi_*$ to equations \ref{eq:fIsClosed} and \ref{eq:ProductEq2} leads to
$$\Delta f_0 = h_0 - \delta_D, \ \ d_M f_0=0,$$
where $f_0 = \int_{S^1} f \in \Omega^2(M)$. So in order to prove that $D$ is diviso, i.e. the current $\delta_D$ is of type $(1,1)$ we just need to show that $\Delta f_0$ is of type $(1,1)$, as $h_0$ certainly is as it is harmonic and $H^{2,0}(M)=0$ for any irreducible Calabi-Yau. Note that such statement does not follow neither from the K\"ahler identities, neither from the $\partial \overline{\partial}$-lemma, as we do not know the type decomposition of $f_0$ neither whether it is exact. However, if $\mathcal{G}$ is a monopole gerbe, then $d^* H_0 \wedge \psi= \ast d \phi$ where $\phi$ is the function such that $\pi_1(H_0)= \phi (d\theta \wedge \omega - \Omega_1)$. Then, this equation yields
$$d\theta \wedge \left( \frac{\partial f}{\partial \theta} \wedge \Omega_2 - d_M^* f \wedge \frac{\omega^2}{2} \right) - \frac{\partial f}{\partial \theta} \wedge \frac{\omega^2}{2}=\ast_M \frac{\partial \phi}{\partial \theta}-d\theta \wedge \ast_M d_M \phi,$$
and applying $\pi_*$ to it gives $d_M^* f_0 \wedge \frac{\omega^2}{2} = - \ast_M D_M \Phi$, where $\Phi= \int_{S^1} \phi$. This equation is easily seen to be equivalent to $d_M^* f_0 = I d_M \Phi$ and so
$$\Delta f_0 = 2 i \partial \overline{\partial} \Phi,$$
which is clearly of type $(1,1)$ and so $D$ is a divisor.
\end{proof}

\section{Some Problems and Questions}

We now state some directions and conjectures for further research along these lines. The first of these goes back to Hitchin's work on the moduli of special Lagrangian submanifolds, \cite{Hitchin99}.

\begin{enumerate}
\item The map $MPic(X, \varphi) \rightarrow H^3(X, \mathbb{Z})$, from theorem \ref{thm:2} has a kernel which can be used to define an analogue of the Jacobian
$$
\begin{array}{ccc}
& CDiv(X, \varphi) & \\
& m\downarrow & \\
Jac(X, \varphi) \rightarrow & MPic(X, \varphi) & \xrightarrow{c_1} H^3(X, \mathbb{Z}) 
\end{array},
$$
In this way $MJac(X, \varphi) \subseteq H^2(X, \mathbb{Z})$ and this can be used to define linear equivalence of coassociatives. Namely, two coassociatives $N_1$ and $N_2$ are linearly equivalent if and only if the holonomies of the flat monopole gerbe associated with $N_1-N_2$ vanish.
\item Is there something like a "section" of a gerbe (the answer seems to be no).
\item Given a $G_2$-manifold $(X, \varphi)$, we can say it is polarizable (or prequantizable) if $[\varphi] \in H^3(X, \mathbb{Z})$, i.e. is integral. In this case, by fixing an good open cover $\lbrace U_i \rbrace_{i \in I}$, we can define a monopole gerbe $(\mathcal{G},F)$ with curvature $H$. This is somewhat similar in spirit to a polarized complex manifold, or a prequantizable sympletic manifold. It remains to see what can we do with this $G_2$-version.
\item In the case of a complex manifold and a complex line bundle $L$, the standard integrability theorem guarantees that if $A$ is a connection on $L$ whose curvature has no $(0,2)$ component, then the bundle is holomorphic, i.e. there are trivializations where the $(0,1)$-components of the connection forms vanish. We conjecture that the following analogue of this works for closed $G_2$-structures.

\begin{problem}\label{conj:integrability}
Let $U \subset \mathbb{R}^7$ be a contractible precompact set, equipped with a closed $G_2$-structure $\varphi$ and $H$ be a closed $3$-form on $U$ with $H_7=0$. Is there a $2$-form $F$ and a function $\phi$, such that $H=dF$ and $F \wedge \psi = \ast d \phi$?
\end{problem}
\end{enumerate}

\end{document}